\documentclass{article}
\usepackage[utf8]{inputenc}

\pdfoutput=1
\usepackage{mathtools}
\usepackage{amsfonts}
\usepackage{amssymb}
\usepackage{centernot}
\usepackage{graphicx}
\usepackage{amsthm}
\usepackage{enumerate}

\usepackage{tikz}
\usepackage{tikz-cd}
\usetikzlibrary{decorations.markings}
\usetikzlibrary{arrows}
\usetikzlibrary{calc}
\providecommand{\U}[1]{\protect\rule{.1in}{.1in}}
\newtheorem{theorem}{Theorem}

\newtheorem{lemma}[theorem]{Lemma}

\newtheorem{proposition}[theorem]{Proposition}

\theoremstyle{remark}

\newcommand{\Z}{\mathbb{Z}}

\allowdisplaybreaks

\begin{document}

\title{A note on longitudes of virtual knots}
\author{Daniel S. Silver
\and
Lorenzo Traldi}

\date{ }
\maketitle

\begin{abstract}
It is a famous property that a longitude of a classical knot lies in the second commutator subgroup of the knot group. We observe that the same property holds for a longitude of a virtual knot.

\emph{Keywords}: Alexander module; meridian; longitude; second commutator subgroup; virtual knot.

Mathematics Subject Classification 2020: 57K12
\end{abstract}

\section{Introduction}\label{intro}

We use fairly standard knot-theoretic terminology regarding virtual knots and diagrams. In particular, an \emph{arc} of a diagram extends from one underpassing point to another. (These are sometimes called ``long arcs.'') If $D$ is an oriented virtual knot diagram then the (Wirtinger) \emph{group} $G(D)$ is given by a presentation with a generator $g_a$ for each arc $a$ of $D$, and a relation $g_bg_a^w=g_a^w g_c $ for each classical crossing of $D$, as pictured in Fig. \ref{fig1}. 

\begin{figure} [tbh]
\centering 
\begin{tikzpicture} [>=angle 90]
\draw [thick] [->] (1,.5) -- (-0.8,-0.4);
\draw [thick] (-1,-.5) -- (-0.8,-0.4);
\draw [thick] (-1,.5) -- (-0.8,.4);
\draw [thick] [->] (-.2,.1) -- (-0.8,0.4);
\draw [thick] (0.2,-0.1) -- (1,-.5);
\node at (1.3,0.5) {$a$};
\node at (-1.3,0.6) {$c$};
\node at (1.3,-0.5) {$b$};

\draw [thick]  (6,.5) -- (4.4,-0.3);
\draw [thick] [<-] (4.4,-0.3) -- (4,-.5);
\draw [thick] (4,.5) -- (4.2,0.4);
\draw [thick] [->] (4.8,0.1) -- (4.2,0.4);
\draw [thick] (5.2,-0.1) -- (6,-.5);
\node at (6.3,0.5) {$a$};
\node at (3.7,0.6) {$c$};
\node at (6.3,-0.5) {$b$};

\node at (0,-1) {$w=-1$};
\node at (5,-1) {$w=1$};
\end{tikzpicture}
\caption{Crossings of writhe $-1$ and $1$.}
\label{fig1}
\end{figure}
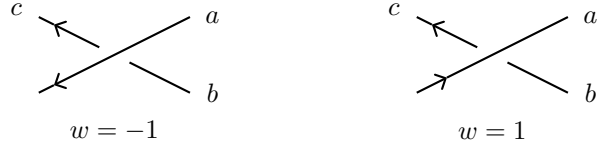

If $b$ is an arc of $D$ then the group element $g_b$ is the \emph{meridian} associated with $b$. The \emph{longitude} associated with $b$ is the group element $\ell_b$ constructed as follows. Starting in the middle of $b$, walk along $K$ in the direction of its orientation. Construct a word in the generators by appending, at each crossing, the monomial $g_a^w$, where $a$ is the overpassing arc and $w$ is the writhe of the crossing, as indicated in Fig.\ \ref{fig1}. At the end, append a power $g_b^p$ so that the resulting word has exponent sum $0$.

The longitudes have the following property.

\begin{lemma}\label{commlem}
For every arc $b_0$ of $D$, the group elements $g_{b_o}$ and $\ell_{b_o}$ commute.
\end{lemma}

\begin{proof} Index the arcs of $D$ as $b_0, \dots, b_{n-1},b_n=b_0$ in order, according to the orientation of $D$. For each index $i \in \{0, \dots, n-1\}$, let $a_i$ be the overpassing arc that separates $b_i$ from $b_{i+1}$, and let $w_i$ be the writhe of that crossing. The defining relations of $G(D)$ give us  $g_{b_i}g_{a_i}^{w_i}=g_{a_i}^{w_i}g_{b_{i+1}}$ for $0 \leq i <n$, so 
\begin{align*}
g_{b_o} \ell_{b_o} &= g_{b_o} g_{a_0}^{w_0}g_{a_1}^{w_1} g_{a_2}^{w_2}\dots g_{a_{n-1}}^{w_{n-1}}=g_{a_0}^{w_0} g_{b_1} g_{a_1}^{w_1} g_{a_2}^{w_2}\dots g_{a_{n-1}}^{w_{n-1}}\\
&=g_{a_0}^{w_0}  g_{a_1}^{w_1} g_{b_2} g_{a_2}^{w_2}\dots g_{a_{n-1}}^{w_{n-1}} = \dots = g_{a_0}^{w_0} g_{a_1}^{w_1} g_{a_2}^{w_2}\dots g_{b_{n-1}} g_{a_{n-1}}^{w_{n-1}}\\
&=g_{a_0}^{w_0} g_{a_1}^{w_1} g_{a_2}^{w_2}\dots  g_{a_{n-1}}^{w_{n-1}}g_{b_{n}} = \ell_{b_o} g_{b_o} .
\qedhere \end{align*} \end{proof}

It is a well-known fact that a longitude of a classical knot diagram $D$ lies in the second commutator subgroup $G(D)''$. The standard proof of this fact involves a Seifert surface; see \cite[p.\ 39]{BZ03}, for instance. Not all virtual knot diagrams have Seifert surfaces, so it might seem possible that a non-classical virtual knot diagram could have a longitude outside of $G(D)''$. Indeed, in July of 2026 Google's AI feature (mistakenly) asserted that this possibility can be realized, and that examples have appeared in a paper of Kim \cite{K}. Surprised by the apparent contradiction of Traldi's result that a longitude of a virtual knot is 0 in the Alexander module \cite[Cor. 11]{T}, and concerned that some technical difference in definitions might have created a discrepancy, we investigated and came to the following conclusion:
\begin{theorem} \label{main}
    The longitudes of a virtual knot diagram $D$ must lie in $G(D)''$.
\end{theorem}

We provide two short proofs of Theorem \ref{main} below.

\section{First proof}\label{first}

Our first proof of Theorem \ref{main} is completely algebraic. We include some details for the convenience of the reader, but we should make it clear that the lemmas are standard parts of the theory of knot groups and modules.

The elementary proof of Lemma \ref{module} is left as an exercise.

\begin{lemma}\label{module}
Let $G$ be an arbitrary group, with commutator subgroup $G'$ and abelianization $H=G/G'$. Then the abelian group $G'/G''$ is a module over $H$ with scalar multiplication defined as follows. If $g \in G$ and $x \in G'$ then $$ (gG') \cdot (xG'') = gxg^{-1}G''.$$
\end{lemma}

Lemma \ref{module} allows us to manipulate elements of $G'/G''$ using conventional notation for module operations. This must be done with care, remembering that the notation does not apply to arbitrary elements of $G/G''$. For instance the calculation
\[
[x,y] = xyx^{-1}y^{-1}=x+y+x^{-1}+y^{-1}=x+y-x-y=x-x+y-y=0
\]
is perfectly valid for $x,y\in G'/G''$, but not valid for general elements of $G/G''$. 

Now, suppose $D$ is a virtual knot diagram. 

\begin{lemma}\label{zero}
    The commutator subgroup $G(D)'$ includes every element of $G(D)$ that can be expressed as a product $\Pi g_{a_i} ^{n_i}$ with $\sum n_i = 0$.
\end{lemma}
\begin{proof}
    A product $g_a g^{-1}_b$ lies in $G(D)'$ because the fact that $g_b$ is a conjugate of $g_a$ implies that $g_a g^{-1}_b$ is a commutator: if $g_b= x g_a x^{-1}$ then $g_a g^{-1}_b = g_a x g_a^{-1} x^{-1} = [g_a,x]$. A product $g_a^{-1} g_b$ is equal to $g_a^{-1} \cdot g_b g_a^{-1}  \cdot g_a$, a conjugate of the commutator $g_b g_a^{-1}$. Proceeding inductively, a longer product can be conjugated to yield a new product in which the last two terms are of opposite sign. Then this new product is equal to $xy$, where $x$ and $y$ are separate shorter products, each with exponent sum $0$.
\end{proof}

It is easy to see that the abelianization $H=G(D)/G(D)'$ is an infinite cyclic group, generated by an element $t$ that is equal to $g_bG(D)'$ for every arc $b$ of $D$. The integral group ring $\Z H$ is conventionally identified with the ring $\Lambda=\Z[t,t^{-1}]$ of Laurent polynomials in the variable $t$, and the $\Lambda$-module $\Lambda \oplus G(D)'/G(D)''$ is the \emph{Alexander module} of $D$. Using module notation, we have the following.

\begin{lemma}\label{auto}
    If $D$ is a virtual knot diagram then multiplication by $t-1$ defines an automorphism of $G(D)'/G(D)''$.
\end{lemma}
\begin{proof}
    Observe that if $a$ and $b$ are arcs of $D$ then the element of $G(D)'/G(D)''$ defined by the commutator of $g_a$ and $g_b$ is an element of $(t-1) \cdot G(D)'/G(D)''$:
    \begin{align*}[g_a,g_b]G(D)'' & = g_a g_b g_a^{-1} g_b^{-1}G(D)''  = g_a  (g_b g_a^{-1} g_b^{-1} g_a) g_a^{-1}G(D)'' \\ & =  t \cdot (g_b g_a^{-1} g_b^{-1} g_aG(D)'')= t \cdot (g_b g_a^{-1}G(D)'') + t \cdot ( g_b^{-1} g_aG(D)'')  \\ & = t \cdot (g_b g_a^{-1}G(D)'') +  g_b( g_b^{-1} g_a)  g_b^{-1}G(D)'' \\ & = t \cdot (g_b g_a^{-1}G(D)'') +  g_a  g_b^{-1}G(D)'' \\  &= t \cdot (g_b g_a^{-1}G(D)'') +(g_b g_a^{-1})^{-1}G(D)'' \\ & = t \cdot (g_b g_a^{-1}G(D)'') - g_b g_a^{-1}G(D'') = (t-1) \cdot (g_b g_a^{-1}G(D)'').\end{align*}

    The commutators $ [g_a,g_b]$ generate $G(D)'$ as a normal subgroup of $G(D)$. That is, the group $G(D)'$ is generated by the commutators $ [g_a,g_b]$ and their conjugates. It follows that the abelian group $G(D)'/G(D)''$ is generated by the images of  the commutators $ [g_a,g_b]$ and their conjugates. Every conjugation is a conjugation by some product of the generators $g_a$; in the module $G(D)'/G(D)''$ this is equivalent to scalar multiplication by some power of $t$. We conclude that $G(D')/G(D)''$ is generated, as a module, by the various elements  $ [g_a,g_b]G''$. As shown in the first paragraph, these elements are all contained in the submodule $(t-1)\cdot G(D)'/G(D)''$. It follows that $(t-1)\cdot G(D)'/G(D)''$ is the entire module $G(D)'/G(D)''$.

Multiplication by a fixed scalar defines an endomorphism of any module over a commutative ring. It follows that multiplication by $t-1$ defines a surjective endomorphism of $G(D)'/G(D)''$. As $\Z[t,t^{-1}]$ is a Noetherian ring, the finitely generated module $G(D)'/G(D)''$ is Noetherian too. Therefore a surjective endomorphism is an automorphism. \end{proof}

\begin{proposition}\label{longann}
    Suppose $y \in G(D)'$ and $D$ has an arc $b$ such that $y G(D)''$ commutes with $g_b G(D)''$ in $G(D)/G(D)''$. Then $y \in G(D)''$.
\end{proposition}

\begin{proof}
    We have $g_b y g_b^{-1} G(D)'' = y G(D)''$, and hence $ t \cdot (y G(D)'') = y G(D)''$; therefore $(t-1)\cdot (yG(D)'') = 0.$ Multiplication by $t-1$ defines an automorphism of $G(D)'/G(D)''$, so $yG(D)''=0$ in $G(D)'/G(D)''$.
\end{proof}

Theorem \ref{main} follows from Proposition \ref{longann} along with Lemmas \ref{commlem}, \ref{zero} and \ref{auto}.

\section{Second proof}\label{secondp}
We adapt a well-known homological argument of \cite{Mi} (see also \cite{Go}) to prove:

\begin{proposition} \label{second} Let $G$ be a finitely presented group with infinite cyclic abelianization $G/G'$. Assume that  $xG'$ generates $G/G'$,  for some $x \in G$. If an element $y \in G'$ commutes in $G$ with $x$, then $y$ is contained in $G''$. \end{proposition} 

Theorem \ref{main} follows from Proposition \ref{second} and Lemma \ref{commlem} by letting $G$ be the virtual knot group, $x$ the class $g_b$ of a meridian and $y$ a longitude associated with $b$. 

\begin{proof} Let $X$ be the standard 2-complex with fundamental group $G$, and let $X_\infty$ be the covering complex with fundamental group $G'$. 
The deck transformation group of $X_\infty$ is infinite cyclic with generator corresponding to $x$. We denote its induced action on the chain complex $C_*(X_\infty)$ and homology groups $H_*(X_\infty)$ both by $t$. (Integer coefficients are assumed here and throughout.) On $H_1(X_\infty) \cong G'/G''$, the deck transformation $t$ corresponds to conjugation in $G$ by $x$. 

The short exact sequence of chain complexes 
$$0 \to C_*(X_\infty) \xrightarrow{t-1} C_*(X_\infty) \to C_*(X) \to 0$$
induces a long exact sequence in homology: 
$$\to H_1(X_\infty)  \xrightarrow{t-1} H_1(X_\infty )\to H_1(X) \to H_0(X_\infty)\xrightarrow{t-1} H_0(X_\infty)\to$$
which can be rewritten: 
$$\to H_1(X_\infty)  \xrightarrow{t-1} H_1(X_\infty )\to \Z \to \Z \xrightarrow{0} \Z \to $$
From this we see as we did in first proof that $t-1$ is surjective. Again since $H_1(X_\infty)$ is a finitely generated $Z[t, t^{-1}]$-module and $\Z$ is Noetherian, $t-1$ is also injective. Now our hypothesis that $y$ commutes in $G$ with $x$ implies that the class of $y$ in $H_1(X_\infty)\cong G'/G''$ is mapped trivially by $t-1$. Hence $y$ is contained in $G''$. 
\end{proof} 

\section{Conclusion}\label{discussion}

Theorem \ref{main} suggests that the algebraic properties of virtual knot longitudes resemble those of classical knot longitudes more closely than may have been thought, as both lie in $G(D)''$. However, the longitudes of classical knots have homological properties that are not required for the longitudes of virtual knots.  See \cite{K} for discussion.

For links, the analogue of the property that the longitudes of a knot lie in $G(D)''$ is that appropriately defined link longitudes sum to $0$ in the reduced Alexander module. Unlike the knot property, the link property does seem to require Seifert surfaces, as it fails for many non-classical virtual links. See \cite{periadd} for discussion.

\section*{ORCID} 

Daniel S. Silver https://orcid.org/0000-0003-1506-3525

\noindent Lorenzo Traldi https://orcid.org/0000-0003-1097-2818

\noindent Department of Mathematics and Statistics\\
\noindent University of South Alabama\\ Mobile, AL 36688 USA\\
\noindent Email: silver@southalabama.edu
\bigskip

\noindent Department of Mathematics\\
\noindent Lafayette College\\Easton, PA 18042 USA\\
\noindent Email: traldil@lafayette.edu\\ 


\begin{thebibliography}{99}

\bibitem{BZ03} G. Burde and H. Zieschang, {\sl Knots}, 2nd ed., Walter de Gruyter, Berlin, 2003.


\bibitem{Go} C.McA. Gordon, Some aspects of classical knot theory, in: {\sl Knot Theory}, Proc. Plans-sur-Bex, Switzerland, 1977 (ed. J.C. Hausmann), LNM 685, Springer-Verlag, Berlin, 1978, 1--60. 

\bibitem{K} S.-G. Kim, Virtual knot groups and their peripheral structure, {\it J. Knot Theory Ramifications} {\bf 9} (2000) 797-812.


\bibitem{Mi} J. Milnor, Infinite cyclic coverings, in: {\sl Conference on the Topology of Manifolds}, Prindle, Weber, and Schmidt, Boston MA 1968, 1151--33. 


\bibitem{periadd} D. S. Silver and L. Traldi, Peripheral elements in reduced Alexander modules: an addendum, {\it J. Knot Theory Ramifications} {\bf 33} (2024) 2471001.


\bibitem{T} L. Traldi, Peripheral elements in reduced Alexander modules,  {\sl J. Knot Theory Ramifications \bf 31} (2022), 2250058.


\end{thebibliography}
\end{document}